\documentclass{baustms}

\usepackage{amsfonts}
\usepackage{latexsym}
\usepackage{amsmath}
\usepackage{amssymb}
\usepackage{amscd}
\usepackage{epsfig}
\usepackage{graphics}

\citesort

\theoremstyle{cupthm}
\newtheorem{thm}{Theorem}[section]

\newtheorem{lemma}[thm]{Lemma}
\theoremstyle{cupdefn}

\theoremstyle{cuprem}

\numberwithin{equation}{section}
\newtheorem{conj}[thm]{Conjecture}

\begin{document}

\runningtitle{Crossing number of cartesian product of Sunlet and Star graphs}
\title{On the crossing number of the cartesian product of a Sunlet graph and a Star graph}
\cauthor 
\author[1]{Michael Haythorpe}
\address[1]{1284 South Road, Clovelly Park, 5042\\
College of Science and Engineering\\
Flinders University, Australia\email{michael.haythorpe@flinders.edu.au}}
\author[2]{Alex Newcombe}
\address[2]{1284 South Road, Clovelly Park, 5042\\
College of Science and Engineering\\
Flinders University, Australia\email{alex.newcombe@flinders.edu.au}}

\authorheadline{M. Haythorpe and A. Newcombe}

\begin{abstract}The exact crossing number is only known for a small number of families of graphs. Many of the families for which crossing numbers have been determined correspond to cartesian products of two graphs. Here, the cartesian product of the Sunlet graph, denoted $\mathcal{S}_n$, and the Star graph, denoted $K_{1,m}$, is considered for the first time. It is proved that the crossing number of $\mathcal{S}_n \Box K_{1,2}$ is $n$, and the crossing number of $\mathcal{S}_n \Box K_{1,3}$ is $3n$. An upper bound for the crossing number of $\mathcal{S}_n \Box K_{1,m}$ is also given.\end{abstract}



\maketitle

\section{Introduction}

Consider a simple graph $G$ with vertex set $V(G)$ and edge set $E(G)$. A {\em drawing} is an embedding of the graph in the plane, in the sense that each vertex $v \in V(G)$ is assigned coordinates in the plane, and each edge $e \in E(G)$ is drawn as a curve starting and ending at the coordinates of its endpoints. A {\em good drawing} is one in which edges have at most one point in common, no more than two edges cross at a single point, and edges which share an endpoint do not cross. For a given drawing $D$ of the graph $G$, the {\em crossings} in the drawing, denoted $cr_D(G)$ can then be computed as the number of times two edges intersect at points other than at their endpoints. The {\em crossing number} $cr(G)$ of a graph $G$ is the smallest number of crossings over all possible drawings of $G$. It is well known that any drawing of $G$ which contains $cr(G)$ crossings is a good drawing.

The {\em crossing number problem}, being the problem of determining the crossing number of a graph, is known to be NP-hard \cite{gareyjohnson}; this is true even for graphs constructed by adding a single edge to a planar graph \cite{cabello}. Indeed, the crossing number problem is known to be notoriously difficult, and is still unsolved even for very small instances. For example, the crossing number of $K_{13}$ has still not been determined, although it is known to be either 223 or 225 \cite{mcquillan,abrego}. However, the crossing number has been determined for some infinite families of graphs. In many such cases, the family is created by taking the cartesian product of members of two smaller graph families. To the best of the authors' knowledge, the first such result published was due to Harary et al \cite{harary} in 1973, who conjectured that the crossing number of $C_m \Box C_n$, that is, the cartesian product of two arbitrarily large cycles, would be $(m-2)n$ for $n \geq m \geq 3$. To date, this conjecture remains unproven, although a number of partial results have been determined. Specifically, the conjecture is known to be true for $m \leq 7$, and also whenever $n \geq m(m-1)$ \cite{ringeisen,beineke,deanrichter,richterthomassen,anderson,anderson2,glebsky}. Other infinite graph families, for which the crossing number of their cartesian products have been studied, include Paths and Stars \cite{jendrolscerbova,klesc1991,bokal}, Complete graphs and Cycles \cite{wenping}, Cycles and Stars \cite{jendrolscerbova,klesc1991}, Wheels and Trees \cite{klesc2017}, and Cycles with the 2-power of Paths \cite{klesc2012}.

In the present work, we expand this growing literature by considering the cartesian product of a Sunlet graph and a Star graph. The Sunlet graph on $2n$ vertices, denoted $\mathcal{S}_n$ for $n \geq 3$, is constructed by attaching $n$ pendant edges to an $n$-cycle $C_n$; see Figure \ref{fig-sunlet} for an example of $\mathcal{S}_6$. The Star graph on $m+1$ vertices consists of one vertex of degree $m$ attached to $m$ vertices of degree 1. It is usually denoted $S_m$, however to avoid confusion with the notation for the Sunlet graph, we note that the Star graph is equivalent to the complete bipartite graph $K_{1,m}$, and use that notation instead. We will show that $cr(\mathcal{S}_n \Box K_{1,m}) \leq n\frac{m(m-1)}{2}$ for $n \geq 3$ and $m \geq 1$. We will also prove that the crossing number meets this bound precisely for $m = \{1, 2, 3\}$, and conjecture that it does so for all $m \in \mathbb{Z}+$.

\begin{figure}[h!]
\centering
    \includegraphics[width=0.175\textwidth]{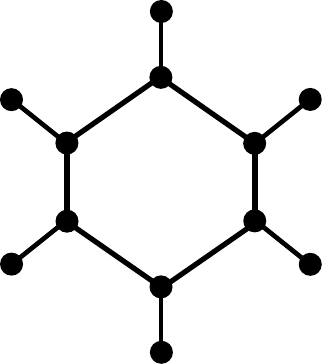}
\caption{The Sunlet graph, $\mathcal{S}_6$.\label{fig-sunlet}}
\end{figure}

\section{Upper Bound}\label{sec-upper}

We begin by providing an upper bound for $cr(\mathcal{S}_n \Box K_{1,m})$. In what follows, let the vertex labels of $K_{1,m}$ be $v_0$ for the vertex of degree $m$ and $v_1, v_2, \hdots, v_m$ for the vertices of degree 1. Let the vertex labels of $\mathcal{S}_n$ be $u_0, u_1, u_2, \hdots, u_{n-1}$ for the vertices on the cycle and let $u_i'$ denote the pendant vertex attached to $u_i$.

\begin{thm}The crossing number of $\mathcal{S}_n \Box K_{1,m}$ is no larger than $n\frac{m(m-1)}{2}$ for $n \geq 3$, $m \geq 1$.\label{thm-upper}\end{thm}

\begin{proof}It is easy to check that $\mathcal{S}_n \Box K_{1,1}$ is planar; for instance, a planar drawing of $\mathcal{S}_6 \Box K_{1,1}$ is illustrated in Figure \ref{fig-K11}, which can obviously be extended for any $n$. It then suffices to give a procedure for drawing the graph $\mathcal{S}_n \Box K_{1,m}$, $m \geq 2$, so that the number of crossings meets the proposed upper bound.

\begin{figure}[h!]
\centering
    \includegraphics[width=0.225\textwidth]{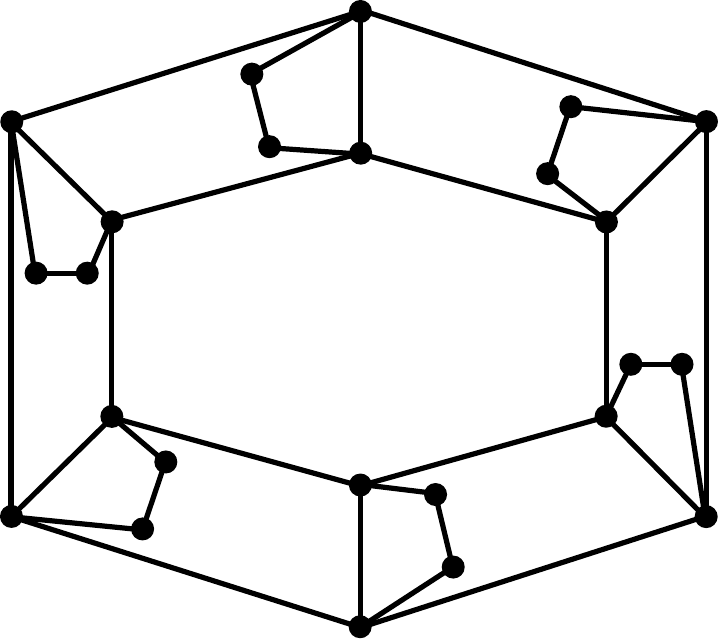}
\caption{Planar drawing of $\mathcal{S}_6 \Box K_{1,1}$.\label{fig-K11}}
\end{figure}

First, note that $\mathcal{S}_n \Box K_{1,m}$ contains $C_n \Box K_{1,m}$ as a subgraph. Begin by drawing the subgraph $C_n \Box K_{1,m}$ in the manner illustrated in Figure \ref{fig-CnK1m}(a). For a given $i = 0, 1, \hdots, n-1$, the thick edges represent $((v_0,u_i),(v_j,u_i))$ for $j = 0, 1, \hdots, m$. The dashed edges represent $((v_j,u_i),(v_j,u_{i+1}))$ and $((v_j,u_i),(v_j,u_{i-1}))$ for $j = 0, 1, \hdots, m$. Then, it is easy to see that the dashed edges can be joined to the corresponding sections for $i+1$ and $i-1$ to complete a drawing of $K_{1,m} \Box C_n$ without introducing any additional crossings. Hence, the number of crossings in this drawing of the subgraph $C_n \Box K_{1,m}$ is:

\begin{equation}
\label{consteq1}
\begin{aligned}
n\Bigg( \sum_{k=1}^{\left \lfloor \frac{m}{2} \right \rfloor-1} k + \sum_{k=1}^{\left \lceil \frac{m}{2} \right \rceil -1} k \Bigg) = n\left \lfloor \frac{m}{2} \right \rfloor \left \lfloor \frac{m-1}{2} \right \rfloor.
\end{aligned}
\end{equation}

\begin{figure}[h!]
\centering
    \includegraphics[width=0.7\textwidth]{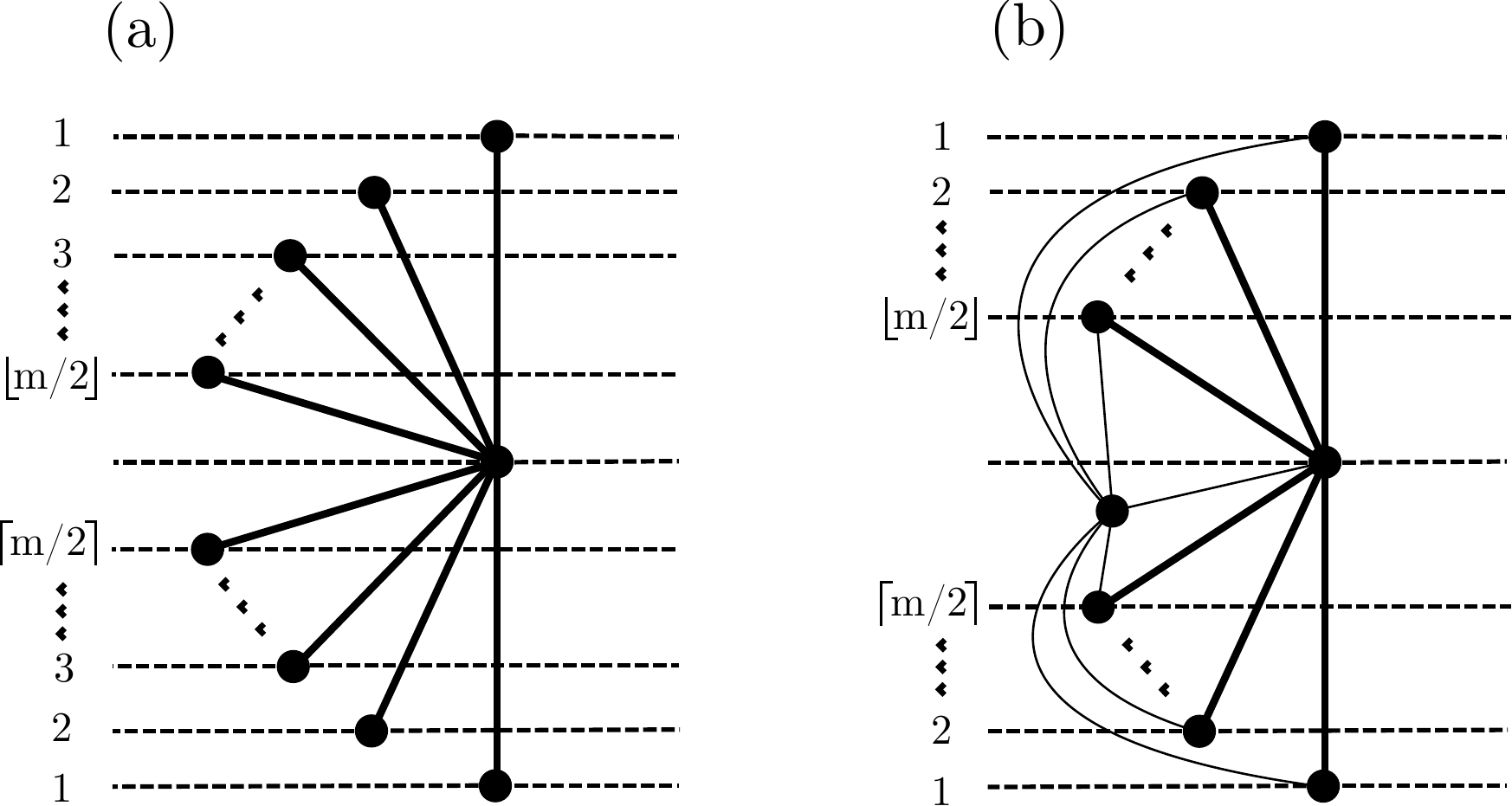}
\caption{In (a), the construction of a drawing of the subgraph $K_{1,m} \Box C_n$.  In (b), the extension which will be subdivided to produce a drawing of $K_{1,m} \Box \mathcal{S}_n$\label{fig-CnK1m}}
\end{figure}

Next, we extend this drawing to a drawing of $\mathcal{S}_n \Box K_{1,m}$ in the following way. For each $i = 0, 1, \hdots, n-1$, place a vertex in the region between the centre horizontal (dashed) edge $((v_0,u_i),(v_0,u_{i+1}))$ and the first thick edge on the side which possesses $\left\lceil m/2 \right\rceil$ vertices, and join this new vertex to each of the vertices $(v_j, u_i)$ for $j = 0, 1, \hdots, m$ as in Figure \ref{fig-CnK1m}(b). Then, the number of crossings in this graph is equal to:

\begin{equation}\label{eq-upper}
\begin{aligned}
n\Bigg(\left \lfloor \frac{m}{2} \right \rfloor \left \lfloor \frac{m-1}{2} \right \rfloor + \sum_{k=1}^{\left \lfloor \frac{m}{2} \right \rfloor} k + \sum_{k=1}^{\left \lceil \frac{m}{2} \right \rceil -1} k \Bigg)\\
=n \left \lfloor \frac{m}{2} \right \rfloor\Bigg( \left \lfloor \frac{m-1}{2} \right \rfloor + \left \lceil \frac{m}{2} \right \rceil \Bigg) \\
=n \frac{m(m-1)}{2}.
\end{aligned}
\end{equation}

Finally, if every new edge is subdivided, except for the ones emanating from $(v_0,u_i)$ for $i = 0, 1, \hdots, n-1$, the resulting graph is isomorphic to $\mathcal{S}_n \Box K_{1,m}$. Since subdividing edges does not alter the number of crossings, we conclude that it is possible to draw $\mathcal{S}_n \Box K_{1,m}$ with $n\frac{m(m-1)}{2}$ crossings.
\end{proof}

\section{Exact results}

We now consider $\mathcal{S}_n \Box K_{1,m}$ for some small values of $m$, and show that the crossing number coincides precisely with the upper bound from Section \ref{sec-upper}. Denote that upper bound by $U(n,m) := n\frac{m(m-1)}{2}$. As noted previously, $\mathcal{S}_n \Box K_{1,1}$ is planar; see Figure \ref{fig-K11}. This agrees with $U(n,1) = 0$. Next, we will consider the cases when $m = 2$ and $m = 3$.

In what follows, we will utilise some properties of subgraphs of $\mathcal{S}_n \Box K_{1,m}$, which we denote by $H_i$ for each $i = 0,1,2,\hdots,n-1.$ In particular, $H_i$ is defined as the subgraph induced by the union of the following, disjoint, sets of edges:

\begin{itemize}
\item[] $a_i := \{\big((v_j,u_i),(v_j,u_{i+1})\big) \mid j=0,1,\dots,m\}$
\item[] $b_i := \{\big((v_j,u_i),(v_j,u'_{i})\big) \mid j=0,1,\dots,m\}$
\item[] $b'_i := \{\big((v_0,u'_i),(v_j,u'_{i})\big) \mid j=1,\dots,m\}$
\item[] $c_i := \{ \big((v_j,u_i),(v_j,u_{i-1})\big) \mid j=0,1,\dots,m\}$
\item[] $t_{i} := \{ \big((v_0,u_{i}),(v_j,u_{i})\big) \mid j=1,\dots,m\}$
\item[] $t_{i+1} := \{ \big((v_0,u_{i+1}),(v_j,u_{i+1})\big) \mid j=1,\dots,m\}$
\item[] $t_{i-1} := \{ \big((v_0,u_{i-1}),(v_j,u_{i-1})\big) \mid j=1,\dots,m\}$
\end{itemize}
A detailed illustration of $H_i$, for the case $m=3$, is displayed in Figure \ref{fig-Hi}.  For each $i=0,1,2,\dots,n-1$, there is a corresponding $H_i$ in $\mathcal{S}_n \Box K_{1,m}$ and $H_i$ and $H_j$ contain common edges when $j=i+1$ or $j=i-1$.

\begin{figure}[h!]
\centering
    \includegraphics[width=0.5\textwidth]{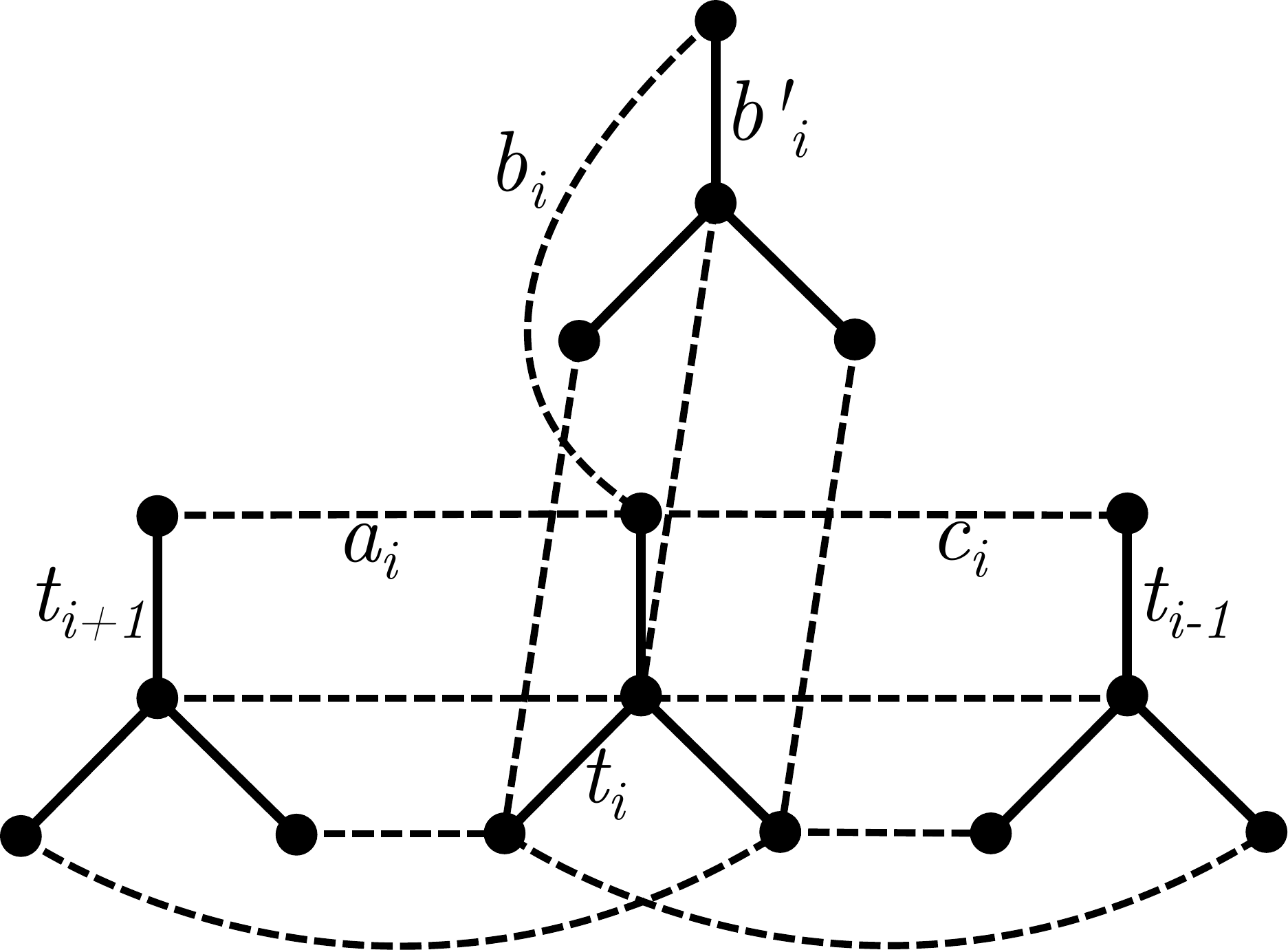}
\caption{The subgraph $H_i$ of $\mathcal{S}_n \Box K_{1,3}$.  The labels for each set of edges lie next to one edge belonging to that set. In this drawing, the thick lines correspond to the sets $t_{i-1}, t_i, t_{i+1}$ and $b'_i$.\label{fig-Hi}}
\end{figure}

We now consider the case when $m = 2$. Note that $U(n,2) = n$. In what follows, we use the following notation: consider a drawing $D$ of a graph which contains two edge sets $a$ and $b$.  Then $cr_D(a)$ is equal to the number of crossings on the edges of $a$ in $D$.  Similarly, $cr_D(a,b)$ is equal to the number of crossings in $D$ between edge-pairs, such that one edge is contained in $a$ and the other is contained in $b$.

\begin{lemma}The crossing number of $\mathcal{S}_n \Box K_{1,2}$ is equal to $n$.\end{lemma}

\begin{proof}From Theorem \ref{thm-upper}, we know that $cr(\mathcal{S}_n \Box K_{1,2}) \leq n$. Hence, the task now is to show that the reverse inequality holds. Let $H'_i$ be the subgraph $H_i$ without the edges $t_i$. An illustration of $H'_i$ is displayed in Figure \ref{fig-H'i}.

\begin{figure}[h!]
\centering
    \includegraphics[width=0.3\textwidth]{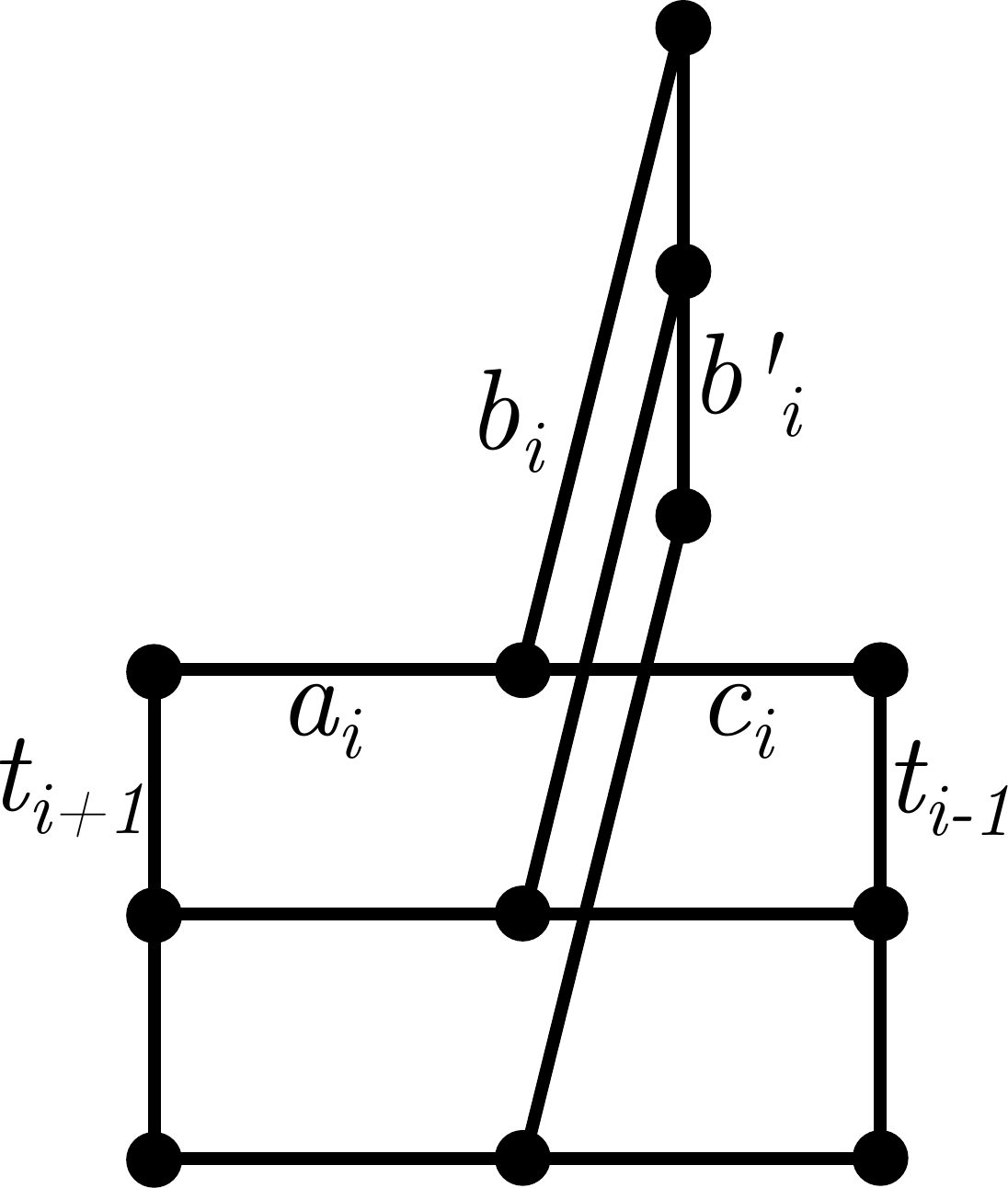}
\caption{The subgraph $H'_i$ of $\mathcal{S}_n \Box K_{1,2}$.  The labels for each set of edges lie next to one edge belonging to that set.\label{fig-H'i}}
\end{figure}

It is clear that $H'_i$ is homeomorphic to $K_{3,3}$, and so there exists at least one crossing in the subdrawing $D'$ of $H'_i$. Furthermore, at least one crossing in $D'$ involves two edges which come from the edge sets $(a_i \cup t_{i+1}), (b_i \cup b'_i)$ and $(c_i \cup t_{i-i})$, but do not both come from the same edge set. That is,

\begin{equation}
\label{eqsnam1}
cr_{D'}\big((a_i \cup t_{i+1}),(b_i \cup b'_{i})\big) + cr_{D'}\big((a_i \cup t_{i+1}),(c_i \cup t_{i-1})\big) + cr_{D'}\big((b_i \cup b'_{i}),(c_i \cup t_{i+1})\big) \geq 1.
\end{equation}

Hence, it is clear that there is at least one crossing in each $H'_i$ which does not occur in any other $H'_j$ for $i \neq j$, which leads immediately to the result.\end{proof}

Next, we consider the case when $m = 3$. Note that $U(n,3) = 3n$. In order to handle this case, we first need to prove two intermediate results, Lemmas \ref{lem-K133}--\ref{lem-3n}.

\begin{lemma}
\label{lem-K133}
For $m = 3$, consider the following four edge sets: $(a_i \cup t_{i+1})$, $(b_i \cup b'_{i})$, $(c_i \cup t_{i-1})$ and ${t_i}$. Then, in any good drawing of the subgraph $H_i$, there are at least 3 crossings for which the two edges involved in the crossing are not in the same set.
\end{lemma}

\begin{proof}
$H_i$ is homeomorphic to $K_{1,3,3}$, and Asano \cite{asano} proved that $\text{cr}(K_{1,3,3})=3$.  Any drawing of $H_i$ corresponds to some drawing of $K_{1,3,3}$.  Any drawing of $K_{1,3,3}$ has at least three crossings between pairs of edges which are not incident.  These crossings correspond precisely to crossings in the drawing of $H_i$ which satisfy the Lemma.
\end{proof}

\begin{lemma}For $n \geq 3$, let $D$ be a drawing of $\mathcal{S}_n \Box K_{1,3}$. If, for each $i = 0, 1, 2, \hdots, n-1$, the edges $t_i \cup b_i \cup b'_i$ are crossed two or fewer times in $D$, then $D$ has at least $3n$ crossings.\label{lem-3n}\end{lemma}

\begin{proof}
Let $F_i$ denote the edge set $t_i \cup b_i \cup b'_i$. Note that $F_i$ is a subgraph of $H_i$. Then, from Lemma \ref{lem-K133}, we have
\begin{equation}
\label{k331-eq1}
cr_D\big(a_i \cup t_{i+1},F_i\big)+cr_D\big(c_i \cup t_{i-1},F_i\big)+cr_D\big((a_i\cup t_{i+1}),(c_i \cup t_{i-1})\big)+cr_D\big(F_i ,F_i\big) \geq 3.
\end{equation}

Assume that $cr_D(F_i) \leq 2$ for all $i = 0,1,2,\hdots,n-1$.  It will be shown that if $cr_D\big(t_{i+1},F_i\big) \neq 0$, or if $cr_D\big(t_{i-1},F_i\big) \neq 0$, then a contradiction arises.

Suppose that $cr_D\big(t_{i+1},F_i\big) = 1$.  Note that the edges of $b_{i+1}$ link to all of the endpoints of $t_{i+1}$.  Since the subgraph induced by $F_i$ is 2-connected, it is clear that it is impossible to draw $(b_{i+1} \cup b'_{i+1})$ without creating an additional crossing on the edges of $F_i$.  Since the subgraph induced by $F_i \cup c_i \cup t_{i-1}$ is isomorphic to $P_2 \Box K_{1,3}$, where $P_2$ denotes the path graph on $3$ vertices, and $cr(P_2 \Box K_{1,3})=1$ \cite{jendrolscerbova}, the following also holds

$$cr_D\big(c_i \cup t_{i-1},F_i\big)+cr_D\big(F_i,F_i\big) \geq 1.$$

This would imply that $F_i$ is crossed at least three times, but by assumption, $cr_D(F_i) \leq 2$. Hence, $cr_D\big(t_{i+1},F_i\big) \neq 1$. An analogous argument can be made for $t_{i-1}$ which, similarly, implies that $cr_D\big(t_{i-1},F_i\big) \neq 1$ as well.

Suppose that $cr_D\big(t_{i+1},F_i\big) = 2$. Then, since $cr_D(F_i) \leq 2$, it must be the case that $cr_D(F_i,F_i) = 0$, and hence without loss of generality, the subdrawing of the subgraph induced by $F_i$ is equivalent to the drawing displayed in Figure \ref{fig-Fiplanar}.

\begin{figure}[h!]
\centering
    \includegraphics[width=0.3\textwidth]{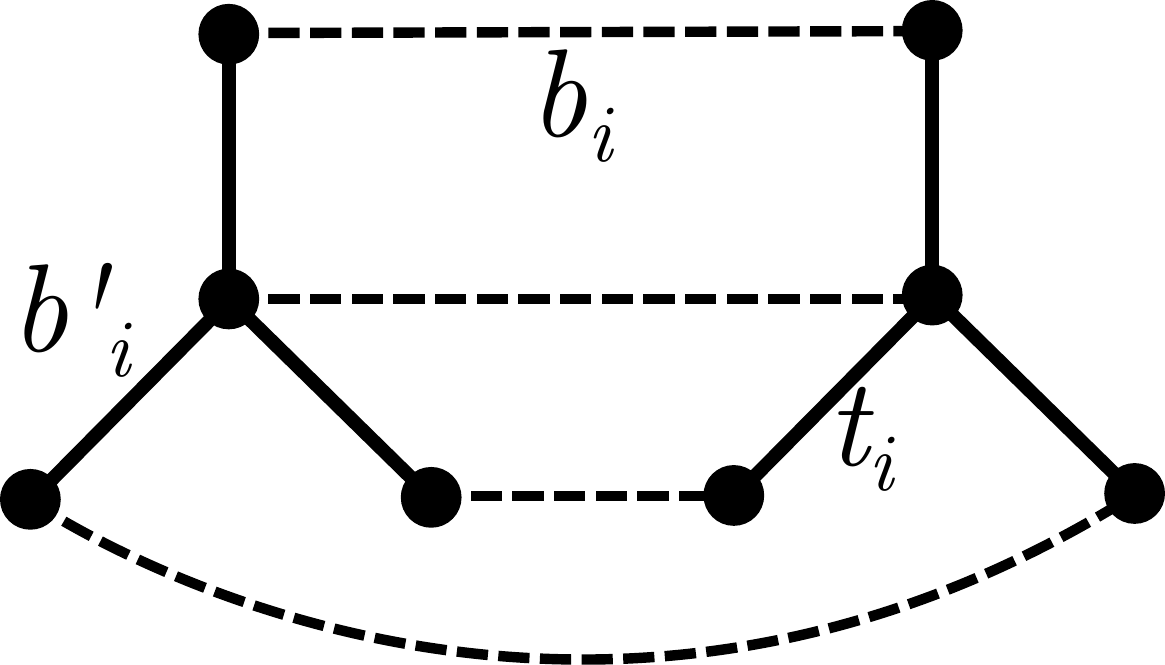}
\caption{The drawing of the subgraph induced by $F_i$, if $F_i$ is not crossed by itself.  \label{fig-Fiplanar}}
\end{figure}

Now consider the rest of the subgraph $H_i$, which includes edge sets $(a_i \cup t_{i+1})$ and $(c_i \cup t_{i-1})$. Note that the edges $c_i$ link to all of the endpoints of $t_i$, and these do not lie on a common face of $D$, so $(c_i \cup t_{i-1})$ cannot be drawn without crossing $F_i$ at least once. This would imply that $F_i$ is crossed at least three times, but by assumption, $cr_D(F_i) \leq 2$.  Hence, $cr_D\big(t_{i+1},F_i\big) \neq 2$. An analogous argument can be made for $t_{i-1}$ which, similarly, implies that $cr_D\big(t_{i-1},F_i\big) \neq 2$ as well.

Then, since $cr_D(F_i) \leq 2$, the only remaining possibility is that $cr_D\big(t_{i+1},F_i\big) = cr_D\big(t_{i-1},F_i\big) = 0$, and so (\ref{k331-eq1}) simplifies to

\begin{equation}
\label{k331-eq2}
cr_D\big(a_i,F_i\big)+cr_D\big(c_i,F_i\big)+cr_D\big((a_i\cup t_{i+1}),(c_i \cup t_{i-1})\big)+cr_D\big(F_i,F_i\big) \geq 3.
\end{equation}

It can be easily seen that any crossing counted by the left hand side of (\ref{k331-eq2}) is not counted for any other $j \neq i$.  Hence summing (\ref{k331-eq2}) over $i=0,1,2,\dots,n-1$ provides the result.
\end{proof}

Finally, we are ready to propose the theorem for $m = 3$.

\begin{thm}For $n \geq 3$, the crossing number of $\mathcal{S}_n \Box K_{1,3}$ is equal to $3n$.\end{thm}

\begin{proof}We will prove the result by induction. The base case where $n = 3$, corresponding to a graph on 24 vertices, was proved computationally, utilising the exact crossing minimisation methods of Chimani and Wiedera \cite{chimani}, which are available at http://crossings.uos.de. The proof comes from a solution to an appropriately constructed integer linear program and shows that $cr(\mathcal{S}_3 \Box K_{1,3}) = 9.$ The proof file is available and can be provided by the corresponding author if desired.

Now, assume that $cr(\mathcal{S}_n \Box K_{1,3}) = 3n$ for each $n = 3, \hdots, k-1$, but that for $n = k$ there exists a drawing with strictly fewer than $3k$ crossings. Let $D$ denote such a drawing. By Lemma \ref{lem-3n}, there must be at least one $i$ such that the edges of $F_i$ are crossed at least three times in $D$. Hence, the edges $F_i$ could be deleted and the number of crossings remaining would be strictly less than $3(k-1)$. However, once $F_i$ is deleted, the resulting graph is homeomorphic to $\mathcal{S}_{k-1} \Box K_{1,3}$, which by the inductive assumption has crossing number equal to $3(k-1)$. This is a contradiction, and hence any drawing for $n=k$ must have at least $3k$ crossings.  This, combined with Theorem \ref{thm-upper} implies that $cr(\mathcal{S}_k \Box K_{1,3}) = 3k$, and inductively we obtain the result. \end{proof}

We conclude by conjecturing that the upper bound described in Theorem \ref{thm-upper} coincides precisely with the crossing number in all cases. To provide evidence supporting this conjecture, we used QuickCross \cite{quickcross}, a recently developed crossing minimization heuristic, to find good drawings of $\mathcal{S}_n \Box K_{1,m}$ for $n,m \leq 20$. In all cases, QuickCross was able to find an embedding that agrees with the conjecture but was never able to find an embedding with fewer crossings.

\begin{conj}For $n \geq 3$, $m \geq 1$,
$$cr(\mathcal{S}_n \Box K_{1,m}) = n\frac{m(m-1)}{2}$$\end{conj}

\end{document}